\documentclass[a4paper,10pt]{scrartcl}

\usepackage[fleqn]{amsmath}
\usepackage{amssymb}
\usepackage{amsthm}
\usepackage{stmaryrd}
\usepackage{bbm}

\newcommand{\bE}{\ensuremath{\mathbb{E}\,}}

\newcommand{\bN}{\ensuremath{\mathbb{N}}}

\newcommand{\bP}{\ensuremath{\mathbb{P}}}

\newcommand{\bR}{\ensuremath{\mathbb{R}}}

\newcommand{\bX}{\ensuremath{\mathbb{X}}}
\newcommand{\bY}{\ensuremath{\mathbb{Y}}}
\newcommand{\bZ}{\ensuremath{\mathbb{Z}}}

\newcommand{\ind}{\ensuremath{\mathbbm{1}}}

\newcommand{\cB}{\ensuremath{\mathcal{B}}}
\newcommand{\cC}{\ensuremath{\mathcal{C}}}

\newcommand{\fF}{\ensuremath{\mathfrak{F}}}

\newcommand{\fP}{\ensuremath{\mathfrak{P}}}

\newcommand{\hqed}{\hfill\qed}

\newcommand{\abs}[1]{\left\vert \, #1 \, \right\vert}
\newcommand{\norm}[1]{\left\Vert \, #1 \, \right\Vert}

\newcommand{\normb}[1]{\interleave \, #1 \, \interleave}

\newcommand{\menge}[2]{\left\{ #1 \;\middle|\; #2 \right\} }

\newcommand{\ddx}[1][1]{\ifnum#1=1 \frac{d}{dx} \else \frac{d^{#1}}{dx^{#1}} \fi}
\newcommand{\ddy}[1][1]{\ifnum#1=1 \frac{d}{dy} \else \frac{d^{#1}}{dy^{#1}} \fi}
\newcommand{\ddt}[1][1]{\ifnum#1=1 \frac{d}{dt} \else \frac{d^{#1}}{dt^{#1}} \fi}

\newcommand{\erwsymbol}{\mathbb{E}}

\newcommand{\erwc}[2]{\erwsymbol\left[#1\,\middle|\,#2\right]}

\newcommand{\suml}{\sum\limits}

\newcommand{\intl}{\int\limits}
\newcommand{\liml}{\lim\limits}
\newcommand{\supl}{\sup\limits}
\newcommand{\infl}{\inf\limits}
\newcommand{\toinf}{\rightarrow\infty}

\newtheorem{theorem}{Theorem}[section]
\newtheorem{lemma}[theorem]{Lemma}
\newtheorem{proposition}[theorem]{Proposition}
\newtheorem{corollary}[theorem]{Corollary}
\newtheorem{definition}[theorem]{Definition}

\renewenvironment{proof}[1][Proof]{\begin{trivlist}
\item[\hskip \labelsep {\bfseries #1}]}{\end{trivlist}}

\newenvironment{remark}[1][Remark]{\begin{trivlist}
\item[\hskip \labelsep {\bfseries #1}]}{\end{trivlist}}

\newcommand{\Lip}{{Lip}}

\title{Concentration of Additive Functionals for Markov Processes and Applications to Interacting Particle Systems}
\author{Frank Redig\thanks{Delft University of Technology , F.H.J.Redig@tudelft.nl} \and Florian V\"ollering\thanks{Universit\"at Leipzig, voellering@math.uni-leipzig.de}}

\renewcommand{\intl}{\int}
\newcommand{\exdomain}{domain }
\DeclareMathOperator{\exdom}{dom}

\begin{document}

\maketitle

\begin{abstract}
We consider additive functionals of Markov processes in continuous time with general (metric) state spaces. We derive concentration bounds for their exponential moments and moments of finite order. Applications include diffusions, interacting particle systems and random walks. The method is based on coupling estimates and not spectral theory, hence reversibility is not needed. We bound the exponential moments(or the moments of finite order) in terms of a so-called coupled function difference, which in turn is estimated using the generalized coupling time. Along the way we prove a general relation between the contractivity of the semigroup and bounds on the generalized coupling time. 
\end{abstract}
\text{\bfseries Keywords:} Markov processes, Polish state space, additive functionals, coupling, generalized coupling time, concentration estimates, exclusion process

\text{\bfseries AMS classification:} 60J25, 60J55, 60F10

\section{Introduction}
The study of concentration properties of additive functionals of Markov processes is the subject of many recent publications, see e.g. \cite{CATTIAUX:GUILLIN:08}, \cite{WU:YAO:08}. This subject is strongly connected to functional inequalities such as the Poincar\'{e} and log-Sobolev inequality, as well as to the concentration of measure phenomenon \cite{LEDOUX:01}.
In the present paper we consider concentration properties of a general class of additive functionals of the form $\int_0^T f_t (X_t) \ dt$ in the context of continuous-time Markov processes on a Polish space. The simplest and classical case is where $f_t=f$ does not depend on time. However the fact that time-dependent functions $f_t$ are allowed can be a significant advantage in applications.

Our approach is based on coupling ideas. More precisely, we estimate exponential moments or $k$-th order moments using the so-called coupled function difference which is estimated in terms of a so-called generalized coupling time, a generalization of the concept used in \cite{CHAZOTTES:REDIG:09}. Because of this approach no knowledge about a possible stationary distribution is required.

Our method covers several cases such as diffusion processes, jump processes, random walks and interacting particle systems. The example of random walk shows that for unbounded state spaces, the concentration inequalities depend on which space the functions $f_t$ belong to. 

The main application to the exclusion process, which has slow relaxation to equilibrium and therefore does not satisfy any functional inequality such as e.g.\ log-Sobolev (in infinite volume), shows the full power of the method. Besides, we give a one-to-one correspondence between the exponential contraction of the semigroup and the fact that the generalized coupling time is bounded by the metric. For discrete state spaces, this means that the semigroup is exponentially contracting if and only if the generalized coupling time is bounded.

Our paper is organized as follows: in Section \ref{section:concentration} we prove our concentration inequalities in the general context of a continuous-time Markov process on a metric space. We derive estimates for exponential moments and moments of finite order. In Section \ref{section:generalized-coupling-time} we study the generalized coupling time and its relation to contractivity of the semigroup. Section \ref{section:examples} is devoted to examples. Section \ref{section:exclusion} deals with the symmetric exclusion process. 

\section{Concentration inequalities}\label{section:concentration}
Let $\bX=(X_t)_{t\geq0}$ be a Feller process in the Polish state space $E$. Denote by $\bP_x$ its associated measure on the path space of cadlag trajectories $D_{[0,\infty[}(E)$ started in $x\in E$ and with 
\[ \fF_t := \sigma\left\{X_s ; 0\leq s \leq t\right\} ,\quad t\geq0,\]
the canonical filtration. We denote by $\bE_x$ the expectation with respect to the measure $\bP_x$. For $\nu$ a probability measure on $E$, we define $\bE_{\nu} := \int \bE_x \,\nu(dx)$, i.e. expectation in the process starting from $\nu$. The associated semigroup we denote by $(S_t)_{t\geq0}$ and with $A$ its generator, both considered on a suitable space ($\cB(E), \cC(E), \cC_0(E),...$).

The content of this section is to derive concentration inequalities for functionals of the form
\begin{align}\label{eq:F-definition}
F(\bX):= \intl_0^\infty f_t(X_t) \,dt, \quad f_t: E\to\bR.
\end{align}
The most familiar case is when $F$ is of the form
\[ \int_0^T f(X_t)\,dt, \]
i.e. $f_t \equiv f$ for $t\leq T$ and $f_t \equiv 0$ for $t>T$.
We first formulate conditions on the family of functions $f_t$ which we will need later.
\begin{definition}
We say the family of functions $\{f_t,t\geq0\}$ is $k$-regular for $k\in\bN$, if:
\begin{enumerate}
\item The $f_t$ are Borel measurable and $t\mapsto f_{t+s}(X_s)$ is Lebesgue-integrable $\bP_x$-a.s. for every $x\in E, t\geq0,$ and $\bE_x \int_0^\infty \abs{f_{t+s}(X_s)} \,ds <\infty$; 
\item $\bE_x \supl_{0\leq s\leq \epsilon}\abs{f_{t+s}(X_s)}^k$ is well-defined and finite for $t\geq0$, $x\in E$ arbitrary and $\epsilon>0$ small enough;
\item There exists a function $r: E \rightarrow \bR$ and $\epsilon_0>0$ such that for $0<\epsilon<\epsilon_0$ and $x\in E$
\[ \supl_{t\geq0}\bE_x\int_0^\infty \abs{f_{t+\epsilon+s}(X_s)-f_{t+s}(X_s)}\,ds \leq \epsilon r(x) \]
and $\bE_x r(X_\epsilon)^k < \infty$.
\end{enumerate}
\end{definition}
\begin{remark}
If $F(\bX)=\int_0^T f(X_t)\,dt$, then $\bE_x \supl_{0\leq t\leq T+\epsilon_0}\abs{f(X_t)}^k < \infty$ for some $\epsilon_0>0$ implies conditions b) and c) of the $k$-regularity.
In condition b) the statement of well-definedness can be replaced by the existence of a measurable upper bound.
\end{remark}
The technique to obtain concentration inequalities for functionals of the form \eqref{eq:F-definition} is to use a telescoping approach where one conditions on $\fF_t$, i.e., where we average $F(\bX)$ under the knowledge of the path of the Markov process $\bX$ up to time $t$. 
\begin{definition}
For $0\leq s\leq t$, define the increments
\[\Delta_{s,t} := \bE[F(\bX)|\fF_t] - \bE[F(\bX)|\fF_s]	\]
and the initial increment
\[ \Delta_{\star,0} := \bE[F(\bX)|\fF_0] - \bE_\nu[F(\bX)],	\]
which depends on the initial distribution $\nu$.
\end{definition}
The basic property of the increments is the relation $\Delta_{s,u} = \Delta_{s,t}+\Delta_{t,u}$ for $s<t<u$. Also, we have
\[ \bE[F(\bX)|\fF_T] - \bE_\nu[F(\bX)] = \Delta_{\star,0} + \Delta_{0,T}, \]
where we have to use $\Delta_{\star,0}$ to accommodate for the initial distribution $\nu$. To better work with the increment $\Delta_{s,t}$, we will rewrite it in a more complicated but also more useful way.
\begin{definition}
Given the family of functions $\{f_t:t\geq0\}$, the coupled function difference is defined as
\begin{align*}
\Phi_t(x,y) := \intl_0^\infty S_u f_{t+u} (x) - S_u f_{t+u}(y)\, du.
\end{align*}
\end{definition}
\begin{remark}
We call $\Phi_t$ the coupled function difference because later we will see that we need estimates on $\abs{\Phi_t}$, and for a coupling $\widehat\bE$ of $\bX$ starting in $x$ and $y$ we have the estimate
\[ \Phi_t(x,y) \leq \intl_0^\infty \widehat{\bE}_{x,y}\abs{f_{t+u}(X_u)-f_{t+u}(Y_u)} \,du .\]
\end{remark}
 In the next lemma we express the increments $\Delta_{s,t}$ in terms of the coupled function difference $\Phi_t$.
\begin{lemma}\label{lemma:delta-representation}
\begin{align*}
 \Delta_{s,t} &= \intl_s^t f_u(X_u) - S_{u-s}f_u(X_s) \,du + [S_{t-s}\Phi_t(X_t,\cdot)](X_s).
\end{align*}
\end{lemma}
\begin{proof}
First, we note that
\[ \bE[F(\bX) | \fF_t] = \intl_0^t f_u(X_u) \,du + \intl_t^\infty S_{u-t}f_u(X_t) \,du, \]
and
\[ \bE[F(\bX) | \fF_s] = \intl_0^s f_u(X_u) \,du + \intl_s^t S_{u-s}f_u(X_s) \,du +  \left[S_{t-s} \intl_t^\infty S_{u-t}f_u \,du\right](X_s). \]
Hence, 
\begin{align*}
\Delta_{s,t} &= \bE[F(\bX) | \fF_t] - \bE[F(\bX) | \fF_s]	\\
 &= \intl_s^t f_u(X_u) - S_{u-s}f_u(X_s) \,du + S_{t-s}\left[\intl_t^\infty S_{u-t}f_u(X_t) - S_{u-t}f_u \,du\right](X_s) \\
&=\intl_s^t f_u(X_u) - S_{u-s}f_u(X_s) \,du + [S_{t-s}\Phi_t(X_t,\cdot)](X_s).
\end{align*}
\hqed
\end{proof}
The following lemma is crucial to obtain the concentration inequalities of Theorems \ref{thm:exponential expectation} and \ref{thm:moment expectation} below. It expresses conditional moments of the increments in terms of the coupled function difference.
\begin{lemma}\label{lemma:delta^k-moments}
Fix $k\in \bN$, $k\geq2$. Assume that the family $(f_t)$ is $k$-regular and suppose that $\Phi_t(\cdot, x)^k$ is in the \exdomain of the generator $A$ for all $x \in E$.
Then
\begin{align*}
 \lim_{\epsilon\to0}\frac{1}{\epsilon} \erwc{\Delta_{t,t+\epsilon}^k}{\fF_t} = (A(\Phi_t(\cdot, X_t)^k))(X_t).
\end{align*}
\end{lemma}
\begin{proof}
We will use the following elementary fact repetitively. For $k\geq2$, if $\abs{b_\epsilon} \leq \epsilon \overline{b}_\epsilon$ and $\supl_{0\leq\epsilon\leq\epsilon_0}\bE \overline{b}_\epsilon^k<\infty$, then 
\begin{align}\label{eq:convergence-fact}
\liml_{\epsilon\to0} \frac1\epsilon\bE( a_{\epsilon} + b_{\epsilon})^k = \liml_{\epsilon\to0}  \frac1\epsilon \bE a_{\epsilon}^k .
\end{align}
By Lemma \ref{lemma:delta-representation},
\begin{align*}
 \Delta_{t,t+\epsilon} &= \intl_t^{t+\epsilon} f_u(X_u) - S_{u-t}f_u(X_t) \,du + [S_{\epsilon}\Phi_{t+\epsilon}(X_{t+\epsilon},\cdot)](X_t).
\end{align*}
First, we show that we can neglect the first term. Indeed,
\[ \abs{\intl_t^{t+\epsilon} f_u(X_u) - S_{u-t}f_u(X_t) \,du }
\leq \epsilon\, \supl_{0\leq s\leq\epsilon}\abs{f_{t+s}(X_{t+s})} + \epsilon\bE_{X_t}^\bY \supl_{0\leq s\leq\epsilon} \abs{f_{t+s}(Y_s)}, \]
we can use part b) of the $k$-regularity to apply fact \eqref{eq:convergence-fact} and get
\[ \liml_{\epsilon\to0} \frac1\epsilon \erwc{\Delta_{t,t+\epsilon}^k}{\fF_t} = \liml_{\epsilon\to0} \frac1\epsilon \erwc{\left[S_\epsilon \Phi_{t+\epsilon}(X_{t+\epsilon}, \cdot)\right]^k(X_t)}{\fF_t} .\]
Next, by writing $\Phi_{t+\epsilon} = \Phi_t + (\Phi_{t+\epsilon}-\Phi_t)$, we will show that the difference can be neglected in the limit $\epsilon\to0$. To this end, we observe that
\begin{align*}
\abs{\Phi_{t+\epsilon}(x,y)-\Phi_t(x,y)} &\leq \intl_0^\infty \bE_x\abs{f_{t+\epsilon+u}(X_u)-f_{t+u}(X_u)}\,du \\
	&\quad+ \intl_0^\infty \bE_y^\bX\abs{f_{t+\epsilon+u}(X_u)-f_{t+u}(X_u)}\,du .
\end{align*} 
Part c) of the $k-$regularity condition allows us to invoke fact \eqref{eq:convergence-fact} again to obtain
\[ \liml_{\epsilon\to0} \frac1\epsilon \erwc{\Delta_{t,t+\epsilon}^k}{\fF_t} = \frac1\epsilon \erwc{\left[ S_\epsilon \Phi_t(X_{t+\epsilon}, \cdot)\right]^k(X_t)}{\fF_t} .\]
Finally, to replace $S_\epsilon \Phi_t(X_{t+\epsilon}, \cdot)$ by $\Phi_t(X_{t+\epsilon}, \cdot)$ by applying fact \eqref{eq:convergence-fact} for a third time, we estimate
\begin{align*}
&\abs{[S_\epsilon\Phi_t(y,\cdot)](x) - \Phi_t(y,x)} \\
&\leq \abs{\int_0^\infty S_{u+\epsilon}f_{t+u+\epsilon}(x) - S_{u+\epsilon}f_{t+u}(x) \,du } + \abs{\int_0^\epsilon S_u f_{t+u}(x) \,du} \\
&\leq \bE_x \int_0^\infty \abs{f_{t+u+\epsilon}(X_{u+\epsilon}) - f_{t+u}(X_{u+\epsilon})} \,du + \epsilon \bE_x \sup_{0\leq u \leq \epsilon} f_{t+u}(X_u), 
\end{align*}
where parts b) and c) of the $k$-regularity then provide the necessary estimates. Now, the desired result is immediately achieved:
\begin{align*}
 \lim_{\epsilon\to0}\frac{1}{\epsilon} \erwc{\Delta_{t,t+\epsilon}^k}{\fF_t}
	&= \lim_{\epsilon\to0}\frac{1}{\epsilon} \left[ S_\epsilon \left(\Phi_t(\cdot, X_t)\right)^k\right](X_t)	\\
	&= A\Phi_t(\cdot, X_t)^k(X_t).	
\end{align*}
\qed
\end{proof}
We can now state our first main theorem, which is a bound of the exponential moment of $F(\bX)$ in terms of the coupled function difference $\Phi_t$.
\begin{theorem}\label{thm:exponential expectation}
Assume that for all $k\in \bN$, the $f_t$ are $k$-regular and $\Phi_t(\cdot, x)^k \in \exdom(A)$ for all $x\in E$.
Then, for any distributions $\mu$ and $\nu$ on $E$,
\begin{align*}
\log\bE_\mu\left[e^{F(\bX)-\bE_\nu F(\bX)}\right] \leq \log(c_0) + {\int_0^\infty\supl_{x\in E} \suml_{k=2}^\infty \frac{1}{k!} (A(\Phi_t^k(\cdot, x)))(x)\,dt}, \\
\log\bE_\mu\left[e^{F(\bX)-\bE_\nu F(\bX)}\right] \geq \log(c_0) + {\int_0^\infty\infl_{x\in E} \suml_{k=2}^\infty \frac{1}{k!} (A(\Phi_t^k(\cdot, x)))(x)\,dt},
\end{align*}
where the influence of the distributions $\mu$ and $\nu$ is only present in the factor
\[ c_0 = \int e^{\nu \left( \Phi_0(x, \cdot)\right) } \,\mu(dx) .\]
\end{theorem}
\begin{remark}
If $H_t:E\times E$ is an upper bound on $\abs{\Phi_t}$ and $H_t(x,x)=0$ for all $x\in E$, then the upper bound of the theorem remains valid if $\Phi_t$ is replaced by $H_t$. In particular, if $f_t \equiv f\ind_{t\leq T}$, $H_t:=\abs{\Phi_0}\ind_{t\leq T}$ serves as a good initial estimate to obtain the upper bound
\[ \log\bE_\mu\left[e^{F(\bX)-\bE_\nu F(\bX)}\right] \leq \log(c_0) + {T \supl_{x\in E} \suml_{k=2}^\infty \frac{1}{k!} A\abs{\Phi_0}^k(\cdot, x)(x)}. \]
Further estimates on $\abs{\Phi_0}$ specific to the particular process can then be used without the need to keep a dependence on $t$.
\end{remark}
\begin{proof}
Define
\[ \Psi(t) := \bE_\mu\left[e^{\Delta_{\star,0} + \Delta_{0,t}}\right]. \]
We see that for $\epsilon>0$,
\begin{align*}
 \Psi(t+\epsilon)-\Psi(t) 
 &= \bE_\mu\left(e^{\Delta_{\star,0} + \Delta_{0,t}} \bE\left[ e^{\Delta_{t,t+\epsilon}}-1\,\middle|\,\fF_t\right] \right)	\\
 &= \bE_\mu\left(e^{\Delta_{\star,0} + \Delta_{0,t}} \bE\left[ e^{\Delta_{t,t+\epsilon}}-\Delta_{t,t+\epsilon}-1\,\middle|\,\fF_t\right] \right),
\end{align*}
where we used the fact that $\bE[\Delta_{t,t+\epsilon}|\fF_t]=0$.
Hence, using Lemma \ref{lemma:delta^k-moments}, we can calculate the derivative of $\Psi$:
\begin{align*}
 \Psi'(t) &= \bE_\mu\left(e^{\Delta_{\star,0} + \Delta_{0,t}} \sum\limits_{k=2}^\infty \frac{1}{k!} (A(\Phi_t(\cdot, X_t)^k))(X_t) \right).
\end{align*}
To get upper or lower bounds on $\Psi'$, we move the sum out of the expectation as a supremum or infimum. Just continuing with the upper bound, as the lower bound is analogue,
\begin{align*}
 \Psi'(t) &\leq \Psi(t)\supl_{x\in E} \sum\limits_{k=2}^\infty \frac{1}{k!} (A(\Phi_t^k(\cdot, x)))(x).
\end{align*}
After dividing by $\Psi(t)$ and integrating, we get
\[ \ln\Psi(T) - \ln \Psi(0) \leq \intl_0^T \supl_{x\in E} \sum\limits_{k=2}^\infty \frac{1}{k!}  (A(\Phi_t^k(\cdot, x)))(x) \,dt , \]
which leads to 
\[ \liml_{T\toinf}\Psi(T) = \bE_\mu\left[e^{F(\bX)-\bE_\nu F(\bX)}\right] \leq \Psi(0) e^{\int_0^\infty\supl_{x\in E} \suml_{k=2}^\infty \frac{1}{k!} (A(\Phi_t^k(\cdot, x)))(x)\,dt} .\]
The value of $c_0 = \Psi(0) = \bE_\mu e^{\Delta_{\star,0}}$ is obtained from the identity 
\[ \Delta_{\star,0} = \nu\left( \Phi_0(X_0, \cdot )\right). \hqed\]
\end{proof}
How the bound in Theorem \ref{thm:exponential expectation} can be used to obtain a deviation probability in the most common case is shown by the following corollary.
\begin{corollary}
Assume that $F(\bX) = \int_0^Tf(X_t)\,dt$, the conditions of Theorem \ref{thm:exponential expectation} are satisfied, and $\sup_{x\in E} A\abs{\Phi_0}^k(\cdot,x)(x)\leq c_1 c_2^k$ for some $c_1,c_2>0$.
Then, for any initial condition $x\in E$,
\begin{align*}
\bP_x(F(\bX)-\bE_x F(\bX)>x)\leq e^{\frac{-\frac12(\frac{x}{c_2})^2}{Tc_1+\frac13\frac{x}{c_2}}}.
\end{align*}
\end{corollary}
\begin{proof}
By Markov's inequality,
\begin{align*}
\bP_x(F(\bX)-\bE_x F(\bX)>x) &\leq \bE_x e^{\lambda F(\bX)-\bE_x \lambda F(\bX)} e^{-\lambda x} \\
&\leq e^{T c_1 \sum_{k=2}^\infty \frac{1}{k!}\lambda^k c_2^k-\lambda x}, 
\end{align*}
where the last line is the result from Theorem \ref{thm:exponential expectation}. Through optimizing $\lambda$, the exponent becomes
\[ \frac{x}{c_2}-(T c_1+\frac{x}{c_2})\log(\frac{x}{T c_1 c_2}+1). \]
To show that this term is less than $\frac{-\frac12(\frac{x}{c_2})^2}{Tc_1+\frac13\frac{x}{c_2}}$, we first rewrite it as the following inequality:
\[ \log(\frac{x}{T c_1 c_2}+1) \geq \frac{\frac{\frac12(\frac{x}{c_2})^2}{Tc_1+\frac13\frac{x}{c_2}}+\frac{x}{c_2}}{T c_1+\frac{x}{c_2}}. \]
Through comparing the derivatives, one concludes that the left hand side is indeed bigger than the right hand side. \qed
\end{proof}

In applications one tries to find good estimates of $\Phi_t$. When looking at the examples in Section \ref{section:examples}, finding those estimates is where the actual work lies. In the case where the functions $f_t$ are Lipschitz continuous with respect to a suitably chosen (semi)metric $\rho$, the problem can be reduced to questions about the generalized coupling time $h$, which is defined and discussed in detail in Section \ref{section:generalized-coupling-time}. 
In case that the exponential moment of $F(\bX)-\bE F(\bX)$ does not exist or the bound obtained from Theorem \ref{thm:exponential expectation} is not useful, we turn to moment bounds. This is the content of the next theorem.
\begin{lemma}\label{lemma:predictable-quadratic-variation}
Assume that the $f_t$ are $2$-regular and $\Phi_t^2(\cdot,x)$ is in the \exdomain of the generator $A$. Then the predictable quadratic variation of the martingale $(\Delta_{0,t})_{t\geq0}$ is 
\[ \left<\Delta_{0,\cdot}\right>_t = \intl_0^t A\Phi_s^2(\cdot, X_s)(X_s)\,ds . \]
\end{lemma}
\begin{proof}
 We have, using Lemma \ref{lemma:delta^k-moments} for $k=2$,
\[ \frac{d}{dt} \left<\Delta_{0,\cdot}\right>_t = \liml_{\epsilon\to0} \frac{1}{\epsilon}\erwc{\Delta_{t,t+\epsilon}^2}{\fF_t} = A\Phi_t^2(\cdot, X_t)(X_t) .\hqed \]
\end{proof}
\begin{theorem}\label{thm:moment expectation}
Let the functions $f_t$ be $2$-regular and $\Phi_t^2(\cdot,x)$ in the \exdomain of the generator $A$. Then
\begin{subequations}\begin{align}
 \left(\bE_\mu\abs{F(\bX)-\bE_\nu F(\bX)}^p\right)^{\frac1p} &\leq C_p \left[ \left(\bE_\mu\left(\int_0^\infty A \Phi_t^2(\cdot, X_t)(X_t)\,dt\right)^{\frac{p}{2}}\right)^{\frac1p} \right. \\
	&\quad\left.+ \left(\bE_\mu\left(\supl_{t\geq 0} \abs{\Phi_t(X_t, X_{t-})}\right)^p\right)^{\frac1p}\right]	\\
	&\quad + \left( \int  \abs{\nu \left(\Phi_0(x,\cdot)\right)}^p \,\mu(dx) \right)^{\frac1p}
\end{align}\end{subequations}
where the constant $C_p$ only depends on $p$ and behaves like $p/\log p$ as $p\to\infty$.
\end{theorem}
\begin{proof}
By the triangle inequality,
\[ \left(\bE_\mu\abs{F(\bX)-\bE_\nu F(\bX)}^p\right)^{\frac1p} \leq \left( \bE_\mu \abs{\Delta_{0,\infty}}^p\right)^{\frac1p} + \left( \bE_\mu \abs{\Delta_{\star,0}}^p\right)^{\frac1p}. \]
 Since $(\Delta_{0,t})_{t\geq0}$ is a square integrable martingale starting at 0, a version of Rosenthal's inequality(\cite{YAO-FENG:FAN-JI:03}, Theorem 1)  implies
\[ \left( \bE_\mu \abs{\Delta_{0,T}}^p\right)^{\frac1p} \leq 
 C_p \left[  \left( \bE_\mu \left<\Delta_{0,\cdot}\right>_T^{\frac{p}{2}}\right)^{\frac1p} 
	+ \left( \bE_\mu \supl_{0\leq t \leq T} \abs{\Delta_{0,t}-\Delta_{0,t-}}^p\right)^{\frac1p} \right] .
\]
Applying Lemma \ref{lemma:predictable-quadratic-variation} to rewrite the predictable quadratic variation $\langle\Delta_{0,\cdot}\rangle_T$ and Lemma \ref{lemma:delta-representation} to rewrite $\Delta_{t-,t}$, we end with the first two terms of our claim after letting $T\to\infty$. The last term is just a different way of writing $\Delta_{\star,0}$:
\[ \left( \bE_\mu \abs{\Delta_{\star,0}}^p\right)^{\frac1p} 
= \left( \int \abs{\nu\left( \Phi_0(x,\cdot)\right) }^p\,\mu(dx)\right)^{\frac1p} .
\]
\qed
\end{proof}
Let us discuss the meaning of the three terms appearing on the right hand side in Theorem \eqref{thm:moment expectation}. 
\begin{enumerate}
\item The first term gives the contribution, typically of order $T^{\frac{p}{2}}$, that one expects even in the simplest case of processes with independent increments.

E.g. if $\mu$ is an invariant measure and $F(\bX) = \int_0^T f(X_t)\,dt$, then 
\[ \bE_\mu \left(\int_0^\infty A \Phi_t^2(\cdot, X_t)(X_t)\,dt\right)^{\frac{p}{2}} \leq T^{\frac{p}2} \int \left(A \Phi_0^2(\cdot, x)(x)\right)^{\frac{p}{2}}\,\mu(dx).	\]
In many cases (see examples below), $\int \left(A \Phi_0^2(\cdot, x)(x)\right)^{\frac{p}{2}}\,\mu(dx) $ can be treated as a constant, i.e., not depending on $T$. There are however relevant examples where this factor blows up as $T\to\infty$.
\item The second term measures rare events of possibly large jumps where it is very difficult to couple. If the process $\bX$ has continuous paths, this term is not present. Usually this term is or bounded or is of lower order than the first term as $T\to\infty$.
\item The third term has only the hidden time dependence of $\Phi_0$ on $T$. It measures the intrinsic variation given the starting measures $\mu$ and $\nu$ and it vanishes if and only if $\mu=\nu=\delta_x$.
\end{enumerate}
It is also interesting to note that the estimate is sharp for small $T$: If one chooses $F(\bX)=\frac1T\int_0^T f(X_t)\,dt$ and looks at the limit as $T\to0$, the first two terms disappear and the third one becomes $(\int \abs{f(x)-\nu(f)}^p \,\mu(dx))^{\frac1p}$, which is also the limit of the left hand side.
\section{Generalized coupling time}\label{section:generalized-coupling-time}
In  order to apply the results of Section \ref{section:concentration}  we need estimates on $\Phi_t$. We can obtain these if we know more about the coupling behaviour of the underlying process $\bX$. To characterize this coupling behaviour, we will look at how close we can get two versions of the process started at different points measured with respect to a distance.

Let $\rho : E\times E \rightarrow [0,\infty]$ be a lower semi-continuous semi-metric. With respect to this semi-metric, we define
\[ \norm{f}_\Lip := \inf\menge{ r\geq 0 }{f(x)-f(y) \leq r\rho(x,y) \ \forall\,x,y\in E}, \]
the Lipschitz-seminorm of $f$ corresponding to $\rho$.
Now we introduce the main objects of study in this section.
\begin{definition}\label{def:coupling distance}
\begin{enumerate}
\item The optimal coupling distance at time $t$ is defined as 
\[ \rho_t(x,y) := \inf_{\pi \in \fP(\delta_xS_t, \delta_yS_t)} \int \rho(x',y') \,\pi(dx'dy'),\]
where the infimum ranges over the set of all possible couplings with marginals $\delta_x S_t$ and $\delta_y S_t$, i.e., the distribution of $X_t$ started from $x$ or $y$.
\item The generalized coupling time is defined as
\[ h(x,y):=\intl_0^\infty \rho_t(x,y) \,dt  .\]
\end{enumerate}
\end{definition}
Now that we have introduced the generalized coupling time, as first application we obtain, using the remark following Theorem \ref{thm:exponential expectation}:
\begin{corollary}\label{cor:exponential lipschitz}
Assume the functions $f_t$ are Lipschitz continuous with respect to a semi-metric $\rho$, and that the conditions of Theorem \ref{thm:exponential expectation} hold true. Then
\begin{align*}
\bE_\mu\left[e^{F(\bX)-\bE_\nu F(\bX)}\right] \leq c_0 e^{\suml_{k=2}^\infty \frac{c_k}{k!} \supl_{x\in E}(A(h^k(\cdot, x)))(x)},
\end{align*}
where
\begin{align*}
c_0 &= \int e^{\supl_{t\geq0}\norm{f_t}_\Lip \nu\left(h(x, \cdot)\right)}\,\mu(dx),	\\
c_k &= \intl_0^\infty\supl_{t\geq0}\norm{f_t}_\Lip^k \,dt.
\end{align*}	
In particular, if $f_t \equiv f$ for $t\leq T$ and $f_t\equiv 0$ for $t>T$, then
\begin{align*}
c_0 &\leq \int e^{\norm{f}_\Lip \nu\left( h(x, \cdot)\right)} \,\mu(dx),	\\
c_k &\leq T\norm{f}_\Lip^k .
\end{align*}	
\end{corollary}
\begin{remark}
If $\overline{h}$ is an upper bound on the generalized coupling time $h$ with $\overline{h}(x,x)=0$, then the result holds true with $h$ replaced by $\overline{h}$.
\end{remark}
\begin{proposition}\label{prop:rho_t representation}
The optimal coupling distance $\rho_t$ has the dual formulation
\[ \rho_t(x,y) = \supl_{\norm{f}_\Lip=1} (S_tf(x)-S_tf(y)). \]
\end{proposition}
\begin{proof}
By the Kantorovich-Rubinstein theorem (\cite{VILLANI:03}, Theorem 1.14), we have
\begin{align*}
 \infl_{\pi \in \fP(\delta_x S_t, \delta_y S_t)} \int \rho \,d\pi &= \supl_{\norm{f}_{\Lip}=1}\left[ \int f \,d(\delta_x S_t) - \int f \,d(\delta_y S_t) \right] \\
&= \supl_{\norm{f}_{\Lip}=1}\left[ (S_t f)(x) - (S_tf)(y)\right].
\end{align*}
\qed
\end{proof}
Also, it is easy to see that the semi-metric properties of $\rho$ translate to $\rho_t$ and thereby to the generalized coupling time $h$.
\begin{proposition}
Both the optimal coupling distance $\rho_t$ and the generalized coupling time $h$ are semi-metrics. 
\end{proposition}
\begin{proof}
We only have to prove the semi-metric properties of $\rho_t$, they translate naturally to $h$ via integration. 

Obviously, $\rho_t(x,x)=0$ and $\rho_t(x,y)=\rho_t(y,x)$ is true for all $x,y\in E$ by definition of $\rho_t$. For the triangle inequality, we use the dual representation:
\begin{align*}
	\rho_t(x,y) &= \supl_{\norm{f}_\Lip=1} (S_tf(x)-S_tf(y)) \\
	&= \supl_{\norm{f}_\Lip=1} (S_tf(x)-S_tf(z)+S_tf(z)-S_tf(y)) \leq \rho_t(x,z) +\rho_t(y,z)
\end{align*}
\qed
\end{proof}
\newcommand{\osc}{{osc}}
\newcommand{\tx}{\widetilde{x}}
\newcommand{\ty}{\widetilde{y}}
\newcommand{\tz}{\widetilde{z}}
\newcommand{\oh}{\overline{h}}
\newcommand{\tpi}{\widetilde{\pi}}
A first result is a simple estimate on the decay of the semigroup $S_t$ in terms of the optimal coupling distance. 
\begin{proposition}\label{prop:S_t-decay}
Let $\mu$ be a stationary probability measure of the semigroup $S_t$. Then
\[ \norm{S_tf -\mu(f)}_{L^p(\mu)} \leq \norm{f}_\Lip \left( \int\mu(dx) \left(\int\mu(dy) \rho_t(x,y)\right)^p\right)^{\frac1p} .\]
\end{proposition}
\begin{remark}
When we choose the metric $\rho$ to be the discrete metric $\ind_{x\neq y}$ (a choice we can make even in a non-discrete setting), we can estimate $\rho_t(x,y)$ by $\widehat{\bP}_{x,y}(\tau>t)$, the probability that the coupling time $\tau = \inf\menge{t\geq0}{X_s^1 = X_s^2 \ \forall s\geq t}$ is larger than $t$ in an arbitrary coupling ${\widehat\bP}_{x,y}$ of the Markov process started in $x$ and $y$. In this case, the result of Proposition \ref{prop:S_t-decay}  reads
\[ \norm{S_tf -\mu(f)}_{L^p(\mu)} \leq \norm{f}_\osc \left( \int\mu(dx) \left(\int\mu(dy) \widehat{\bP}_{x,y}(\tau>t)\right)^p\right)^{\frac1p}, \]
where $\norm{f}_{osc}=\sup_{x,y}(f(x)-f(y))$ is the oscillation norm.
Note that this can also be gained from the well-known coupling inequality
\[ \norm{\delta_x S_t - \delta_y S_t}_{TVar} \leq 2\widehat{\bP}_{x,y}(\tau>t) .\]
\end{remark}
\begin{proof}[Proof of Proposition \ref{prop:S_t-decay}]
First,
\begin{align*}
	\abs{S_t f(x) - \mu(f)} &= \abs{S_t f(x) - \mu(S_t f)}	\\
	&= \abs{\bE_x f(X_t) - \int\mu(dy)\bE_y f(Y_t)}	\\
	&\leq \int\mu(dy)\abs{\bE_x f(X_t) - \bE_y f(Y_t)}	\\
	&\leq \int\mu(dy)\norm{f}_\Lip\rho_t(x,y).
\end{align*}
This estimate can be applied directly to get the result:
\begin{align*}
\norm{S_t f -\mu(f)}_{L^p(\mu)} &= \left(\int\mu(dx) \abs{S_t f(x) - \mu(f)}^p\right)^{\frac1p}	\\
	&\leq \norm{f}_\Lip \left( \int\mu(dx) \left(\int\mu(dy) \rho_t(x,y)\right)^p\right)^{\frac1p}.
\end{align*}
\qed
\end{proof}
The above result did not use the semigroup property of $S_t$. When we use it we can improve estimates considerably. The price is that from now on, $\rho$ has to be a metric, and this metric must be compatible with the Markov process, which we will define a little bit later under the notion of contraction with respect to this metric. The aim is to show how the uniform boundedness of the generalized coupling time implies an exponential decay of the semigroup $(S_t)$ in the Lipschitz seminorm. To this end, we need the following lemma:
\begin{lemma}\label{lemma:couplingdistance-semigroup}
Under the condition that $\rho$ is a metric,
\[ \supl_{ x\neq y }\frac{\rho_t(x,y)}{\rho(x,y)} = \norm{S_t}_\Lip .\] 
\end{lemma}
\begin{proof}
By the representation of the optimal coupling distance in Proposition \ref{prop:rho_t representation},
\begin{align*}
 \supl_{x\neq y }\frac{\rho_t(x,y)}{\rho(x,y)} 
&= \supl_{ x\neq y }\supl_{\norm{f}_\Lip=1}\frac{S_tf(x)-S_tf(y)}{\rho(x,y)} \\
&= \supl_{\norm{f}_\Lip=1} \norm{S_tf}_\Lip = \norm{S_t}_\Lip.\qedhere
\end{align*}
\qed
\end{proof}
\begin{definition}
We say that the process $\bX$ acts as a \emph{contraction for the distance} $\rho$ if
\begin{align}\label{eq:contraction condition}
 \rho_t(x,y) \leq \rho(x,y)\quad \forall\,t\geq0, 
\end{align}
or equivalently,
\[ \norm{S_t}_\Lip \leq 1 \quad\forall\,t\geq0. \]
\end{definition}
This property is sufficient to show that the process is contracting the distance monotonely:
\begin{lemma}\label{lemma:contraction monotonicity}
Assume that the process $\bX$ acts as a contraction for the distance. Then
\[ \rho_{t+s}(x,y) \leq \rho_t(x,y)	\quad\forall\,x,y\in E, s,t\geq0 .\]
\end{lemma}
\begin{proof}
Using the dual representation,
\begin{align*}
 \rho_{t+s}(x,y) &= \supl_{\norm{f}_\Lip=1} [S_{t+s}f(x) - S_{t+s}f(y)]	\\
	&= \supl_{\norm{f}_\Lip\leq1} [S_t(S_sf)(x) - S_t(S_sf)(y)] .
\end{align*}
By our assumption, the set of functions $f$ with $\norm{f}_\Lip\leq1$ are a subset of the set of functions $f$ with $\norm{S_sf}_\Lip\leq1$. Hence,
\begin{align*}
 \rho_{t+s}(x,y) &\leq \supl_{f:\norm{S_sf}_\Lip\leq1}[S_t(S_sf)(x) - S_t(S_sf)(y)]	\\
	&\leq \supl_{\norm{g}\leq1}[S_tg(x) - S_tg(y)] = \rho_t(x,y).
\end{align*}
\qed
\end{proof}
With this property in mind, we can show the main theorem of this section.
\begin{theorem}\label{thm:Lipschitz contraction}
Assume that $\rho$ is a metric and that the process $\bX$ acts as a contraction for the distance. Then the fact that the generalized coupling time $h$ is bounded by the metric $\rho$ is equivalent to the fact that the semigroup $(S_t)$ is exponentially contracting. More precisely, for $\alpha>1$ arbitrary,
\begin{enumerate}
\item $\forall\,x,y\in E:\ h(x,y) \leq M\rho(x,y) \quad\Rightarrow\quad \forall\,t\geq M\alpha:\ \norm{S_t}_{\Lip} \leq \frac1\alpha;$
\item $\norm{S_T}_{\Lip} \leq \frac1\alpha\quad\Rightarrow\quad\forall\, x,y\in E:\ h(x,y) \leq \frac{\alpha T}{\alpha-1}\rho(x,y)$.
\end{enumerate}
\end{theorem}
\begin{proof}\par{a)}
For $x, y \in E$, set
\[ T_{x,y} := \inf \menge{t\geq0}{\rho_t(x,y) \leq \frac1\alpha\rho(x,y) } .\]
Then, 
\begin{align*} 
	M\rho(x,y) \geq h(x,y) &= \intl_0^\infty \rho_t(x,y) \,dt \geq \intl_0^{T_{x,y}} \rho_t(x,y) \,dt \geq T_{x,y} \frac1\alpha \rho(x,y).
\end{align*}
Therefore $T_{x,y}$ is bounded by $M\alpha$. By Lemma \ref{lemma:contraction monotonicity}, $\rho_t(x,y) \leq \rho_{T_{x,y}}(x,y)$ for all $t\geq T_{x,y}$.
Hence $ \rho_{M\alpha}(x,y) \leq \frac1\alpha\rho(x,y) $ uniformly, which implies
$\norm{S_{M\alpha}}_\Lip \leq \frac1\alpha$.
\par{b)}
Since $\rho_t(x,y) \leq \rho(x,y)\norm{S_t}_\Lip$,
\begin{align*}
 h(x,y) &= \intl_0^\infty \rho_t(x,y) \,dt \leq \rho(x,y)\intl_0^\infty \norm{S_t}_\Lip \,dt	\\
	&\leq \rho(x,y) T \suml_{k=0}^\infty \norm{S_T}_\Lip^k \leq \frac{\alpha T}{\alpha-1}\rho(x,y).
\end{align*}
\qed
\end{proof}
When we apply this theorem to an arbitrary Markov process where we use the discrete distance, we get the following corollary:
\begin{corollary}
The following two statements are equivalent:
\begin{enumerate}
\item The generalized coupling time with respect to the discrete metric $\rho(x,y)=\ind_{x\neq y}$ is uniformly bounded, i.e. 
\[ h(x,y) \leq M \quad \forall x,y\in E; \]
\item The semigroup is eventually contractive in the oscillation (semi)norm, i.e. $\norm{S_T}_{\osc} < 1$ for some $T>0$.
\end{enumerate}
\end{corollary}
\begin{remark}
Theorem \ref{thm:Lipschitz contraction} actually gives us more information, namely how the constants $M$ and $T$ can be related to each other.
\end{remark}
\begin{proof}
Since obviously $\supl_{x\neq y} \rho_t(x,y) \leq 1$, the process $\bX$ acts as a contraction for the discrete distance and the result follows from Theorem \ref{thm:Lipschitz contraction}, where we also use the fact that in the case of the discrete metric, $\norm{\cdot}_\Lip =\norm{\cdot}_\osc$. \qed
\end{proof}
Since Theorem \ref{thm:Lipschitz contraction} part a) implies that $\norm{S_t}_\Lip$ decays exponentially fast, it is of interest to get the best estimate on the speed of decay, which is the content of the following proposition:
\begin{proposition}\label{prop:lipschitz spectral gap}
Assume that $\rho$ is a metric, the process $\bX$ acts as a contraction for the distance and the generalized coupling time $h$ satisfies $h(x,y) \leq M \rho(x,y)$.
Then
\[ \liml_{t\to\infty}\frac1t\log\norm{S_t}_\Lip \leq -\frac1M. \]
\end{proposition}
\begin{proof}
The proof uses the same structure as the proof of part a) in Theorem \ref{thm:Lipschitz contraction}.
First, fix $\epsilon$ between 0 and $\frac1M$. Then define
\[ T_{x,y} = \inf\menge{ t>0 }{ \rho_t(x,y) \leq \rho(x,y) e^{-(\frac1M - \epsilon)t} }. \]
By our assumption, 
\[ M\rho(x,y) \geq h(x,y) \geq \rho(x,y) \intl_0^{T_{x,y}} e^{-(\frac1M - \epsilon)t} \,dt = M\rho(x,y) \frac{1-e^{-(\frac1M - \epsilon)T_{x,y}}}{1-M\epsilon} .\]
Since the fraction on the right hand side becomes bigger than 1 if $T_{x,y}$ is too large, there exists an uniform upper bound $T(\epsilon)$ on $T_{x,y}$. Hence, for all $t\geq T(\epsilon)$, $\norm{S_t}_\Lip \leq e^{-(\frac1M - \epsilon)t}$, which of course implies $\liml_{t\toinf}\frac1t\norm{S_t}_\Lip \leq -\frac1M + \epsilon$. By sending $\epsilon$ to 0, we finish our proof. \qed
\end{proof}
Again, we apply this result to the discrete metric to see what it contains.
\begin{corollary}\label{cor:uniformcoupling}
Let $\widehat{\bP}_{x,y}$ be a coupling of the process $\bX$ started in $x$ resp. $y$, and denote with $\tau:=\inf\menge{t\geq0}{X_s^1 = X_s^2 \ \forall s\geq t}$ the coupling time. Set $M:=\supl_{x,y\in E} {\widehat\bE}_{x,y}\tau$. Then 
\[ \liml_{t\toinf}\frac1t\log\norm{S_t}_\osc \leq -\frac{1}{M}. \]
Equivalently, for $f\in L^\infty$,
\[ \liml_{t\toinf}\frac1t\log\norm{S_t f - \mu(f)}_\infty \leq -\frac{1}{M}, \]
where $\mu$ is the unique stationary distribution of $\bX$.
\end{corollary}
\begin{remark}[Remarks]
\begin{enumerate}
\item If the the Markov process $\bX$ is also reversible, then the above result extends to $L^1$ and hence to any $L^p$, where the spectral gap is then also at least of size $\frac1M$.
\item As an additional consequence, when a Markov process can be uniformly coupled, i.e. $\supl_{x,y\in E}\widehat\bE_{x,y} \tau \leq M<\infty$ for a coupling $\widehat\bE$, then there exists (a possibly different) coupling $\widetilde\bE_{x,y}$, so that $\supl_{x,y\in E}\widetilde\bE_{x,y}e^{\lambda\tau} <\infty$ for all $\lambda<\frac{1}{M}$. Note that without Corollary \ref{cor:uniformcoupling} this property is obvious only for Markovian couplings.
\end{enumerate}

\end{remark}


\section{Examples}\label{section:examples}

\subsection{Diffusions with a strictly convex potential}
Let $V$ be a twice continuously differentiable function on the real line with $V''\geq c>0$ and $\int e^{-V(x)}dx = Z_V < \infty$. To the potential $V$ is associated the Gibbs measure 
\[ \mu_V(dx) = \frac{1}{Z_V}e^{-V(x)}dx \]
and a Markovian diffusion 
\[ dX_t = -V'(X_t) + \sqrt{2}dW_t \]
with $\mu_V$ as reversible measure. 

To estimate the optimal coupling distance $\rho_t$ at time $t$(see Definition \ref{def:coupling distance}), we couple two versions of the diffusion, $X_t^x$ started in $x$ and $X_t^y$ started in $y$, by using the same Brownian motion $(W_t)_{t\geq0}$. Then the difference process $X_t^x-X_t^y$ is deterministic, $x<y$ implies $X_t^x<X_t^y$ and by the convexity assumption
\begin{align*}
d(X_t^y-X_t^x) = -(V'(X_t^y)-V'(X_t^x)) \leq - c (X_t^y-X_t^x).
\end{align*}
Using Gronwall's Lemma, we obtain the estimate
\[ \rho_t(x,y) \leq \abs{x-y}e^{-ct} \]
on the optimal coupling distance. By integration, the generalized coupling time $h$ has the estimate $h(x,y) \leq \frac{1}{c}\abs{x-y}$. As a consequence, Proposition \ref{prop:lipschitz spectral gap} implies
\[ \liml_{t\to\infty} \log \norm{S_t}_\Lip \leq -c. \]

Since the generator $A$ of the diffusion is
\[ A = \frac{d^2}{dx^2} - V' \cdot \frac{d}{dx}, \]
we have
\[ A \left(\frac1c\abs{\cdot-x}\right)^k(x) = 
 \begin{cases}
  \frac{2}{c^2}, &k=2, \\
  0, &k>2.
 \end{cases}
\]
Therefore, for $f:\bR\to\bR$ be Lipschitz-continuous, we can use Corollary \ref{cor:exponential lipschitz} to get the estimate
\begin{align}\label{eq:example 1}
\bE_{\nu_1}\left[e^{\int_0^T f(X_t)\,dt-\bE_{\nu_2}\int_0^T f(X_t)\,dt}\right] \leq c_{\nu_1,\nu_2} e^{T\frac{\norm{f}_\Lip^2}{c^2}} ,	
\end{align}
with the dependence on the distributions $\nu_1$ and $\nu_2$ given by
\[ c_{\nu_1,\nu_2} = \bE_{\nu_1}^x e^{\bE_{\nu_2}^y\frac{\norm{f}_\Lip}{c} \abs{x-y}} .\]
\begin{remark}
\begin{enumerate}
\item An alternative proof that strict convexity is sufficient for \eqref{eq:example 1} to be true can be found in \cite{VILLANI:09}. A proof via the log-Sobolev inequality can be found in \cite{LEDOUX:01}. Hence the result is of no surprise, but the method of obtaining it is new.
\item This example demonstrates nicely how in the case of diffusions the higher moments of $Ah^k(\cdot, x)(x)$ can disappear because the generalized coupling time is bounded by a multiple of the initial distance.
\item The generalization to higher dimensions under strict convexity is straightforward.
\end{enumerate}
\end{remark}
\subsection{Interacting particle systems}
Let $E=\{0,1\}^{\bZ^d}$ be the state space of the interacting particle system with a generator $L$ given by
\[ Lf(\eta) = \suml_x \suml_{\Delta\subset \bZ^d}c(\eta, x+\Delta) [f(\eta^{x+\Delta})-f(\eta)], \]
where $\eta^{\Delta}$ denotes the configuration $\eta$ with all spins in $\Delta$ flipped. This kind of particle system is extensively treated in \cite{LIGGETT:05}. For $f:E\rightarrow\bR$, we denote with $\delta_f(x):=\supl_{\eta\in E}f(\eta^x)-f(\eta)$ the maximal influence of a single flip at site $x$, and with $\delta_f = (\delta_f(x))_{x\in E}$ the vector of all those influences. 

If there is a way to limit how flips in the configuration affect the system as time progresses, then we can obtain a concentration estimate. Again, denote with $F(\eta_{\cdot})=\int_0^T f(\eta_t)\,dt$ the additive functional of the function $f$ and the particle system $\eta_{\cdot}$.
\begin{theorem}\label{thm:IPS}
Assume there exists a family of operators $A_t$ so that
$ \delta_{S_t f} \leq A_t \delta_f, $
and write
\[ G := \int_0^\infty\! A_t \,dt, \]
which is assumed to exist. Denote with 
\[ c_k := \sup_{\eta\in E,x\in\bZ^d} \sum_{\Delta\subset\bZ^d}c(\eta,x+\Delta)\abs{\Delta}^k \]
the weighted maximal rate of spin flips. If $\norm{G}_{p\to2}<\infty$ for some $p\geq1$, then for any $f$ with $\delta_f \in \ell^p$ and any initial condition $\eta\in E$,
\[ \bE_{\eta} e^{F(\eta_{\cdot}) - \bE_{\eta}F(\eta_{\cdot})} \leq \exp\left[{T\suml_{k=2}^\infty \frac{c_k \norm{G}_{p\to2}^k\norm{\delta_f}_p^k}{k!} }\right]. \]
If additionally $\norm{G}_1<\infty$ and $\normb{f}:=\norm{\delta_f}_1<\infty$, then for any two probability distributions $\nu_1$, $\nu_2$,
\[ \bE_{\nu_1} e^{F(\eta_{\cdot}) - \bE_{\nu_2}F(\eta_{\cdot})} \leq \exp\left[\norm{G}_1\normb{f}+{T\suml_{k=2}^\infty \frac{c_k \norm{G}_{p\to2}^k\norm{\delta_f}_p^k}{k!} }\right]. \]
\end{theorem}
Applications of this Theorem are for example spin flip dynamics in the so-called $M<\epsilon$ regime, where there exists an operator $\Gamma$ with $\norm{\Gamma}_1=M$, so that
\[ \delta_{S_t f} \leq e^{-t(\epsilon-\Gamma)}\delta_f \]
holds. Since $G = \int_0^\infty e^{-t(\epsilon-\Gamma)} \,dt = (\epsilon-\Gamma)^{-1}$, $\norm{G}_1\leq(\epsilon-M)^{-1}$. Hence $\norm{G}_{1\to2}\leq(\epsilon-M)^{-1}$ for a first application of the Theorem. If the process is reversible as well, $\norm{G}_\infty=\norm{G}_1$, and by Riesz-Thorin's Theorem, we have $\norm{G}_2\leq (\epsilon-M)^{-1}$, hence we get the result for functions $f$ with $\norm{\delta_f}_2<\infty$. 

Another example is the exclusion process. As a single discrepancy is preserved and moves like a random walk, $A_t(x,y) = p_t(x,y)$, the transition probability of the random walk. In high dimensions, $G(x,y) = \int_0^\infty p_t(x,y)\,dt$ has bounded $\ell^1\to\ell^2$-norm:
\begin{align*}
\norm{G}_{1\to2} &= \sup_{\norm{g}_1=1}\sum_x (\sum_y G(x,y)g(y))^2 \\
&\leq \sup_{\norm{g}_1=1}\sum_x \sum_y \abs{g(y)} G(x,y)^2 \leq \sum_x G(x,0)^2 \infty \\
& = \int_0^\infty \int_0^\infty \sum_x p_t(0,x)p_s(0,x)\,ds\,dt = \int_0^\infty \int_0^\infty p_{s+t}(0,0)\,ds\,dt < \infty
\end{align*}
in dimension 5 and higher. As the exclusion process switches two sites, $c_k\leq2^k$, and hence 
\[ \bE_{\eta} e^{F(\eta_{\cdot}) - \bE_{\eta}F(\eta_{\cdot})} \leq \exp\left[{T\suml_{k=2}^\infty \frac{2^k \norm{G}_{1\to2}^k\normb{f}^k}{k!} }\right]. \]
However, this is only a quick result exploiting the strong diffusive behaviour in high dimensions. In the last section we will deal with the exclusion process in much more detail to obtain results for lower dimensions as well.
\begin{proof}[Proof of Theorem \ref{thm:IPS}]
First, we notice that the coupled function difference $\Phi_t$ for a single flip can be bounded by
\begin{align*}
\Phi_t(\eta^x,\eta) &\leq \int_0^\infty \abs{S_t f(\eta^x) - S_t f(\eta)}\,dt\\
&\leq \int_0^\infty \delta_{S_t f}(x)\,dt \leq \int_0^\infty (A_t\delta_{f})(x)\,dt\\
&\leq (G \delta_f)(x)
\end{align*}
uniformly in $\eta$.
To estimate the coupled function difference $\Phi_t$ we telescope over single site flips, 
\[ \Phi_t^k(\eta^{x+\Delta},\eta) \leq \abs{\Delta}^k ((G\delta_f)(x))^k,  \] 
and therefore
\begin{align*}
L \Phi_t^k(\cdot, \eta)(\eta) &= \suml_x \suml_{\Delta\subset \bZ^d}c(\eta, x+\Delta) \Phi_t^k(\eta^{x+\Delta},\eta) \\
&\leq \suml_x \suml_{\Delta\subset \bZ^d}c(\eta, x+\Delta) \abs{\Delta}^k(G\delta_f)^k(x) \\
&\leq c_k \norm{G\delta_f}_k^k \leq c_k \norm{G\delta_f}_2^k \leq c_k \norm{G}_{p\to2}^k\norm{\delta_f}_p^k
\end{align*}
Hence the first part is proven by applying these estimates to Theorem \ref{thm:exponential expectation} for fixed and identical initial conditions. To prove the estimate for arbitrary initial distributions, we simply observe that, again by telescoping over single site flips,
\[ \Phi_0(\eta,\xi) \leq \sum_x \sup_\zeta \Phi_0(\zeta^x,\zeta) \leq \sum_x (G \delta_f)(x) \leq \norm{G}_1\norm{f}_1. \]
\qed
\end{proof}
\subsection{Simple symmetric random walk}
The aim of this example is to show that we can get concentration estimates even if the process $\bX$ - in this example a simple symmetric nearest neighbour random walk in $\bZ^d$ - has no stationary distribution.  We will consider three cases: $f\in \ell^1(\bZ^d)$, $\ell^2(\bZ^d)$ and $\ell^\infty(\bZ^d)$, and $F(\bX)=\int_0^T f(X_t)\,dt$. To apply Theorem \ref{thm:exponential expectation}, our task is to estimate $\abs{\Phi_t(x,y)}$ where $y$ is a neighbour of $x$. We will denote with $p_t(x,z)$ the transition probability from $x$ to $z$ in time $t$. We start with the estimate on the coupled function difference
\begin{align*}
\abs{\Phi_t(x,y)} &= \abs{\intl_0^{T-t}\bE_x f(X_s)- \bE_y f(X_s)\,ds} \\
&= \abs{\intl_0^{T-t} \suml_{z\in\bZ^d} f(z)(p_s(x,z)-p_s(y,z))\,ds}	\\
&\leq \suml_z \abs{f(z)}\abs{\intl_0^{T-t}p_s(x,z)-p_s(y,z)\,ds} \\
&\leq \suml_z \abs{f(z)}\abs{\intl_0^T p_s(x,z)-p_s(y,z)\,ds}.
\end{align*}
Now, depending on the three cases of $f$, we proceed differently. First, let $f\in\ell^1$. Then,
\begin{align*}
\abs{\Phi_t(x,y)} &\leq \suml_z \abs{f(z)}\abs{\intl_0^T p_s(x,z)-p_s(y,z)\,ds} \\
&\leq \norm{f}_1\supl_z\abs{\intl_0^T p_s(x,z)-p_s(y,z)\,ds} \\
&= \norm{f}_1 \intl_0^T p_s(0,0)-p_s(y-x,0)\,ds	
\leq C_1 \norm{f}_1.
\end{align*}
Since $\abs{x-y}=1$, the constant $C_1 = \int_0^\infty p_s(0,0)-p_s(y-x,0)\,ds $
depends on the dimension but nothing else.

Second, let $f\in\ell^\infty$. Then,
\begin{align*}
\abs{\Phi_t(x,y)} &\leq \suml_z \abs{f(z)}\abs{\intl_0^T p_s(x,z)-p_s(y,z)\,ds} \\
&\leq \norm{f}_\infty \suml_z\abs{\intl_0^T p_s(x,z)-p_s(y,z)\,ds} \\
&= \norm{f}_\infty \intl_0^T \suml_z \abs{p_s(x,z)-p_s(y,z)}\,ds	\\
&= \norm{f}_\infty \intl_0^T \frac12\norm{p_s(x,\cdot)-p_s(y,\cdot)}_{TVar}\,ds	\\
&\leq \norm{f}_\infty \intl_0^T \widehat{\bP}_{x,y}(\tau>s)\,ds	
\end{align*}
In the last line, we used the coupling inequality. The coupling $\widehat{\bP}_{x,y}$ is the Ornstein coupling, i.e., the different coordinates move independently until they meet. Since $x$ and $y$ are equal in all but one coordinate, the probability of not having succeeded at time $t$ is of order $t^{-\frac12}$.
Hence we end up with
\[ \abs{\Phi_t(x,y)} \leq C_\infty \norm{f}_\infty \sqrt{T}. \]

Third, let $f\in\ell^2$. This is the most interesting case.
\begin{lemma}\label{lemma:p_t estimate}
Let $x,y \in \bZ^d$ be neighbours. Then
\[ \suml_{z\in\bZ^d}\left(\intl_0^T p_t(x,z)-p_t(y,z) \,dt\right)^2 \leq \alpha(T) \]
with 
\begin{align*}
\alpha(T) \in \begin{cases}
O(\sqrt{T}), \quad&d=1;\\
O(\log{T}), \quad&d=2;\\
O(1), \quad&d\geq3.
\end{cases}
\end{align*}
\end{lemma}
\begin{proof}
By expanding the product and using the fact that $\suml_z p_t(a,z)p_s(b,z) = p_{t+s}(a,b)=p_{t+s}(a-b,0)$, we get
\begin{align*}
&\suml_{z\in\bZ^d}\left(\intl_0^T p_t(x,z)-p_t(y,z) \,dt\right)^2 
= 2\intl_0^T\intl_0^T p_{t+s}(0,0)-p_{t+s}(x-y,0) \,dt \,ds	\\
&\quad= 2\intl_0^T\intl_0^T (-\Delta)p_{t+s}(\cdot,0)(0) \,dt \,ds	
= 2\intl_0^T p_{s}(0,0)-p_{T+s}(0,0) \,ds	\\
&\quad\leq 2\intl_0^T p_{s}(0,0) \,ds =:\alpha(T).
\end{align*}
\qed
\end{proof}
Using first the Cauchy-Schwarz inequality and then Lemma \ref{lemma:p_t estimate},
\begin{align*}
\abs{\Phi_t(x,y)}^k \leq \norm{f}_2^k \left(\suml_{z}\left(\intl_0^T p_t(x,z)-p_t(y,z) \,dt\right)^2\right)^{\frac k2} \leq \norm{f}_2^k\alpha(T)^{\frac k2}.
\end{align*}
To conclude this example, we finally use the uniform estimates on $\Phi_t$ to apply Theorem \ref{thm:exponential expectation} and obtain
\begin{align*}
&\bE_x\exp\left[\intl_0^T f(X_t)\,dt - \bE_x\intl_0^T f(X_t)\,dt\right] \leq \exp\left[T\suml_{k=2}^\infty\frac{C_1^k\norm{f}^k_1}{k!}\right],\quad&&f\in\ell^1;\\
&\bE_x\exp\left[\intl_0^T f(X_t)\,dt - \bE_x\intl_0^T f(X_t)\,dt\right] \leq \exp\left[T\suml_{k=2}^\infty\frac{\norm{f}_2^k}{k!}\alpha(T)^{\frac k2}\right],\quad&&f\in\ell^2;\\
\intertext{and}
&\bE_x\exp\left[\intl_0^T f(X_t)\,dt - \bE_x\intl_0^T f(X_t)\,dt\right] \leq \exp\left[T\suml_{k=2}^\infty\frac{C_\infty^k\norm{f}^k_\infty}{k!}T^{\frac k2}\right],&&f\in\ell^\infty.\\
\end{align*}
Since the generator is $A f(x) = \frac{1}{2d} \sum_{y\sim x} (f(y)-f(x))$, we use the estimates $2d$ times and divide by $2d$, so no additional constants appear in the results.

\section{Application: Simple symmetric exclusion process}\label{section:exclusion}
This example is somewhat more involved(because of the conservation law), and shows the full power of our approach in the context where classical functional inequalities such as the log-Sobolev inequality do not hold.

The simple symmetric exclusion process is defined via its generator
\[ Af(\eta) = \suml_{x\sim y} \frac{1}{2d} (f(\eta^{xy})-f(\eta)). \]

It is known that the large deviation behaviour of the occupation time of the origin $\int_0^T \eta_t(0)\,dt$ is dependent on the dimension \cite{LANDIM:92}. Its variance is of order $T^{\frac32}$ in dimension $d=1$, $T\log(T)$ in dimension $d=2$ and $T$ in dimensions $d\geq3$ \cite{ARRATIA:85}. We can reproduce this result dimension $d\geq2$\footnote{A previous version of this paper had a statement also in $d=1$ regarding general quasi-local functions. However, there was an error in the computations, so we removed the statement for $d=1$.} for the occupation time of a finite set $A$. 
Consider the occupation indicator $H_A(\eta) := \prod\limits_{a\in A}\eta(a)$ of a finite set $A\subset\bZ^d$.
\begin{theorem}\label{thm:exclusion d>1}
Let $A \subset \bZ^d$ be a finite, and fix an initial configuration $\eta_0\in \{0,1\}^{\bZ^d}$. Then, for all $\lambda>0$,
\[ \bE_{\eta_0}\exp\left(\int_0^T \lambda H_A(\eta_t)\,dt - \bE_{\eta_0}\int_0^T \lambda H_A(\eta_t)\,dt\right)
\leq e^{T\alpha(T)\suml_{k=2}^\infty \frac{(c \lambda \abs{A}^3)^k}{k!}}, \]
where $\alpha(T) \in O(T^{\frac12}), O(\log T)$ or $O(1)$ in dimensions $d=1$, $d=2$ or $d\geq3$.
The constant $c>0$ is independent of $A$, $\eta_0$ and $T$, but may depend on the dimension $d$.
\end{theorem}
The proof of Theorem \ref{thm:exclusion d>1} are subject of the subsection below.

\newcommand{\hE}{\widehat{\mathbb{E}}}

\subsection{Concentration of the occupation time of a finite set in $d\geq2$: Proof of Theorem \ref{thm:exclusion d>1}}
Now, we want to show that the occupation time $\int_0^T H_A(\eta_t)\,dt, H_A(\eta) := \prod_{a\in A}\eta(a),$ of a finite set $A \subset \bZ^d$ has the same time asymptotic behaviour as the occupation time of a single site. As a stating point to estimate $L\abs{\Phi_t}^k(\cdot,\eta)$, we use the following result of  \cite{FERRARI:GALVES:LANDIM:00}:
\begin{theorem}{\cite{FERRARI:GALVES:LANDIM:00}, Theorem 2.2}
\begin{align*}
&\bE_\eta \prod_{a\in A} \eta_t(a) - \prod_{a\in A}\rho^\eta_t(a)	\\
&\quad= -\frac12\int_0^t\,ds \sum_{\substack{Z\subset \bZ^d\\ \abs{Z}=\abs{A}}}\bP_A(X_s = Z)\sum_{\substack{z_1,z_2\in Z\\z_1\neq z_2}}p(z_1,z_2)(\rho_{t-s}^\eta(z_1) - \rho_{t-s}^\eta(z_2))^2\prod_{\substack{z_3 \in Z\\z_3\neq z_1,z_2}}\rho^\eta_{t-s}(z_3)
\end{align*}
Here $\bP_A(X_s=Z)$ is the probability of exclusion walkers started in $A$ occupying the set $Z$ at time $s$, and $\rho_t^\eta(z)=\bE_\eta \eta_t(z)$ is the occupation probability of $z$ at time $t$ given the initial configuration $\eta$.
\end{theorem}
By using this comparison of exclusion dynamics with independent random walkers, we get
\begin{align*}
&\bE_{\eta^{xy}} \prod_{a\in A} \eta_t(a) -\bE_{\eta} \prod_{a\in A} \eta_t(a) \\
&\quad= \bE_{\eta^{xy}} \prod_{a\in A} \eta_t(a) - \prod_{a\in A}\rho^{\eta^{xy}}_t(a) + \prod_{a\in A}\rho^{\eta^{xy}}_t(a) -\prod_{a\in A}\rho^{\eta}_t(a)+\prod_{a\in A}\rho^{\eta}_t(a)-\bE_{\eta} \prod_{a\in A} \eta_t(a)	\\
&\quad= \left(\prod_{a\in A}\rho^{\eta^{xy}}_t(a) -\prod_{a\in A}\rho^{\eta}_t(a)\right) -\frac12\intl_0^t\,ds \sum_{\substack{Z\subset \bZ^d\\ \abs{Z}=\abs{A}}} \bP_A(X_s = Z) \sum_{\substack{z_1,z_2\in Z\\z_1\neq z_2}}p(z_1,z_2) \cdot\\
&\qquad \cdot\left[(\rho^{\eta^{xy}}_{t-s}(z_1) - \rho^{\eta^{xy}}_{t-s}(z_2))^2 \prod_{\substack{z_3 \in Z\\z_3\neq z_1,z_2}}\rho^{\eta^{xy}}_{t-s}(z_3) -(\rho^{\eta}_{t-s}(z_1) - \rho^{\eta}_{t-s}(z_2))^2 \prod_{\substack{z_3 \in Z\\z_3\neq z_1,z_2}}\rho^{\eta}_{t-s}(z_3)\right]
\end{align*}
Taking absolute values, we start to estimate the first difference:
\begin{align*}
\abs{\prod_{a\in A}\rho^{\eta^{xy}}_t(a) -\prod_{a\in A}\rho^{\eta}_t(a)} 
&\leq \suml_{a\in A} \abs{\rho^{\eta^{xy}}_t(a)-\rho^{\eta}_t(a)} = \suml_{a\in A} \abs{p_t(x,a) -p_t(y,a)}.
\end{align*}
The next part is the big difference inside the integral. It is estimated by
\begin{align*}
&\abs{(\rho^{\eta^{xy}}_{t-s}(z_1) - \rho^{\eta^{xy}}_{t-s}(z_2))^2 - (\rho^{\eta}_{t-s}(z_1) - \rho^{\eta}_{t-s}(z_2))^2} \\
&+ \suml_{\substack{z_3 \in Z\\z_3\neq z_1,z_2}}\abs{\rho^{\eta^{xy}}_{t-s}(z_3)-\rho^{\eta}_{t-s}(z_3)}(\rho^{\eta}_{t-s}(z_1) - \rho^{\eta}_{t-s}(z_2))^2
\end{align*}
Now we come back to the original task of estimating $L\abs{\Phi_t}^k(\cdot,\eta)$. From now on, multiplicative constants are ignored on a regular basis, which results in an omitted factor of the form $c_1 c_2^k$. However warning is given by using $\lesssim$ instead of $\leq$. By using the above estimates, we obtain the upper bound
\begin{align}
&\suml_{x\in\bZ^d}\suml_{y\in\bZ^d}p(x,y)\left(\intl_0^T \suml_{a\in A} \abs{p_t(x,a) -p_t(y,a)}\,dt\right)^k \label{eq:exclusion-A}\\
\begin{split}
&+ \suml_{x\in\bZ^d}\suml_{y\in\bZ^d}p(x,y)\left(\intl_0^T\,dt\intl_0^t\,ds \sum_{\substack{Z\subset \bZ^d\\ \abs{Z}=\abs{A}}} \bP_A(X_s = Z) \sum_{\substack{z_1,z_2\in Z\\z_1\neq z_2}}p(z_1,z_2) \cdot\right.\\
&\qquad\left.\cdot\abs{(\rho^{\eta^{xy}}_{t-s}(z_1) - \rho^{\eta^{xy}}_{t-s}(z_2))^2 - (\rho^{\eta}_{t-s}(z_1) - \rho^{\eta}_{t-s}(z_2))^2}\right)^k 
\end{split}\label{eq:exclusion-B}\\
\begin{split}
&+ \suml_{x\in\bZ^d}\suml_{y\in\bZ^d}p(x,y)\left(\intl_0^T\,dt\intl_0^t\,ds \sum_{\substack{Z\subset \bZ^d\\ \abs{Z}=\abs{A}}} \bP_A(X_s = Z) \sum_{\substack{z_1,z_2\in Z\\z_1\neq z_2}}p(z_1,z_2) \cdot\right.\\
&\qquad\left.\cdot\suml_{\substack{z_3 \in Z\\z_3\neq z_1,z_2}} \abs{\rho^{\eta^{xy}}_{t-s}(z_3)-\rho^{\eta}_{t-s}(z_3)}(\rho^{\eta}_{t-s}(z_1) - \rho^{\eta}_{t-s}(z_2))^2 \right)^k,\label{eq:exclusion-C}
\end{split}
\end{align}
which we will treat individually.

For term \eqref{eq:exclusion-A}, we estimate sum over $A$ by the maximum times $\abs{A}$. Hence
\begin{align*}
\eqref{eq:exclusion-A} &\leq \abs{A}^k\suml_{x\in\bZ^d}\suml_{y\in\bZ^d}p(x,y)\left(\intl_0^T \abs{p_t(x,a_0) -p_t(y,a_0)}\,dt\right)^k.
\end{align*}
We note that
\begin{align*}
&\supl_{x\in\bZ,y\sim x} \intl_0^T \abs{p_t(x,a_0) -p_t(y,a_0)}\,dt \\
&\quad\leq \supl_{\abs{j}=1}\intl_0^\infty p_t(0,0) -p_t(j,0)\,dt<\infty
\end{align*}
and
\begin{align*}
&\suml_{x\in\bZ^d}\suml_{y\in\bZ^d}p(x,y)\left(\intl_0^T \abs{p_t(x,a_0) -p_t(y,a_0)}\,dt\right)^2\\
&\quad=\frac{1}{2d}\suml_{\abs{j}=1}\suml_{x\in\bZ^d}\left(\intl_0^T p_t(x,a_0) -p_t(x,a_0+j)\,dt\right)^2\\
&\quad\leq \alpha(T)
\end{align*}
by Lemma \ref{lemma:p_t estimate}. Hence
\[\eqref{eq:exclusion-A} \lesssim \abs{A}^k\alpha(T). \]

Next, we must treat \eqref{eq:exclusion-B}. In the case $k=1$, 
\begin{subequations}\begin{align}
\eqref{eq:exclusion-B} 
&\lesssim \intl_0^T\,dt\intl_0^t\,ds \suml_{x\in\bZ^d}\suml_{y\sim x}\sum_{z_1\in \bZ^d}\sum_{z_2\sim z_1} \left(\sum_{Z: z_1,z_2\in Z} \bP_A(X_s = Z)\right)\label{eq:exclusion-B-1A}\\
&\qquad\cdot\abs{\rho^{\eta^{xy}}_{t-s}(z_1) - \rho^{\eta^{xy}}_{t-s}(z_2) - \rho^{\eta}_{t-s}(z_1) + \rho^{\eta}_{t-s}(z_2)}\label{eq:exclusion-B-1B}\\
&\qquad\cdot\abs{\rho^{\eta^{xy}}_{t-s}(z_1) - \rho^{\eta^{xy}}_{t-s}(z_2) + \rho^{\eta}_{t-s}(z_1) - \rho^{\eta}_{t-s}(z_2)}.\label{eq:exclusion-B-1C}
\end{align}\end{subequations}
Regarding the exclusion walkers $X_s$ in \eqref{eq:exclusion-B-1A}, we can simplify by using Liggett's correlation inequality (\cite{LIGGETT:05}, chapter 8):
\begin{align*}
\sum_{Z: z_1,z_2\in Z} \bP_A(X_s = Z) &= \bP_A(z_1, z_2 \in X_s) \leq \bP_A(z_1\in X_s)\bP_A(z_2\in X_s) \\
&= \left(\sum_{a\in A} p_s(z_1, a) \right)\left(\sum_{a\in A} p_s(z_2, a) \right).
\end{align*}

\begin{lemma}
For $\abs{i},\abs{j}=1$,
\begin{enumerate}
	\item For any $\eta$,
		\begin{align*} 
		&\abs{\rho^{\eta^{x,x+j}}_t(z) - \rho^{\eta^{x,x+j}}_t(z+i) - \rho^{\eta}_t(z) + \rho^{\eta}_t(z+i)} \\
		&\quad\leq \abs{p_t(x,z)-p_t(x+j,z) - p_t(x,z+i) + p_t(x+j,z+i)}, 
	\end{align*}
	\item $ \suml_{x\in\bZ^d} \abs{p_t(x,z)-p_t(x+j,z) - p_t(x,z+i) + p_t(x+j,z+i)} \lesssim (1+t)^{-1}. $
\end{enumerate}
Part b) holds as well when we sum over $z$ instead of $x$.
\end{lemma}
\begin{proof}
First we notice that
\begin{align*}
\rho_t^{\eta^{xy}}(z) - \rho_t^{\eta}(z) &= 
\begin{cases}
p_t(y,z) - p_t(x,z),\quad&\eta(x)=1, \eta(y)=0;\\
p_t(x,z) - p_t(y,z),\quad&\eta(x)=0, \eta(y)=1;\\
0,&otherwise,
\end{cases}
\end{align*}
which immediately proves a). To show b), 
\begin{align*}
&\suml_{x\in\bZ^d} \abs{p_t(x,z)-p_t(x+j,z) - p_t(x,z+i) + p_t(x+j,z+i)}\\
&\quad= \suml_x \Big|\sum_u p_{t/2}(x,u)p_{t/2}(u,z)-p_{t/2}(x+j,u)p_{t/2}(u,z) \\
&\qquad\qquad- p_{t/2}(x,u)p_{t/2}(u,z+i) + p_{t/2}(x+j,u)p_{t/2}(u,z+i)\Big|\\
&\quad\leq\suml_x \sum_u \abs{ (p_{t/2}(x,u) - p_{t/2}(x+j,u)) (p_{t/2}(u,z)-p_{t/2}(u,z+i))}\\
&\quad = \sum_u \abs{p_{t/2}(u,z)-p_{t/2}(u,z+i)}\sum_x\abs{ p_{t/2}(x,u) - p_{t/2}(x+j,u)}\\
&\quad= 4 \norm{p_{t/2}(0,\cdot)-p_{t/2}(i,\cdot)}_{TVar}\norm{p_{t/2}(0,\cdot)-p_{t/2}(j,\cdot)}_{TVar}\\
&\quad\lesssim (1+t/2)^{-\frac12}(1+t/2)^{-\frac12}\leq 2 (1+t)^{-1} ,
\end{align*}
where the last line relies on optimal coupling of two random walks, see for example \cite{LINDVALL:92}. \qed
\end{proof}

As a third observation, 
\begin{align}
\abs{\rho^{\eta}_t(z_1) - \rho^{\eta}_t(z_2)} &= \abs{\suml_{x}(p_t(z_1,x)-p_t(z_2,x))\eta(x)}\nonumber\\
&\leq \norm{p_t(z_1,\cdot)-p_t(z_2,\cdot)}_{TVar},\label{eq:exclusion-rho-difference}
\end{align}
which leads to the estimate
\[ \eqref{eq:exclusion-B-1C} \leq 2\norm{p_{t-s}(z_1,\cdot)-p_{t-s}(z_2,\cdot)}_{TVar} \lesssim (1+t-s)^{-\frac12}. \]

Applying the estimates for \eqref{eq:exclusion-B-1A} to \eqref{eq:exclusion-B-1C}, we have (for $k=1$)
\begin{align*}
\eqref{eq:exclusion-B} &\lesssim
\intl_0^T\,dt\intl_0^t\,ds \sum_{z_1\in \bZ^d}\sum_{z_2\sim z_1} \left(\sum_{a\in A} p_s(z_1, a) \right)\left(\sum_{a\in A} p_s(z_2, a) \right) (1+t-s)^{-\frac32}	\\
&\leq 2d\abs{A}^2 \intl_0^T\,dt\intl_0^t\,ds\ p_s(0,0) (1+t-s)^{-\frac32} \\
&\lesssim \abs{A}^2 \intl_0^T\,dt\intl_0^t\,ds\ (1+s)^{-\frac{d}{2}} (1+t-s)^{-\frac32} \\
&\lesssim \abs{A}^2 \alpha(T),
\end{align*}
where the last line is due to the following lemma:
\begin{lemma}
\begin{align*}
\int_0^T \,dt\int_0^t\,ds\ (1+s)^{-\frac{n}{2}}(1+t-s)^{-\frac32} \lesssim \begin{cases}
\sqrt{T},\quad&n=1;\\
\log(1+T),&n=2;\\
1,&n\geq3.
\end{cases}
\end{align*}
\end{lemma}
\begin{proof}
Write
\[ f(m,n):= \int_0^t (1+s)^{-\frac{m}{2}}(1+t-s)^{-\frac{n}{2}}\,dt.\]
Then $f$ satisfies $f(m,n) \leq (1+t)^{-\frac12}(f(m-1,n)+f(m,n-1))$ for $m,n\geq1$:
\begin{align*}
f(m,n) &= (1+t)^{-\frac12}\int_0^t \frac{(1+t)^{\frac12}}{(1+s)^{\frac{m}{2}}(1+t-s)^{\frac{n}{2}}}\,dt \\
&\leq (1+t)^{-\frac12}\int_0^t \frac{(1+s)^{\frac12}+(1+t-s)^{\frac12}}{(1+s)^{\frac{m}{2}}(1+t-s)^{\frac{n}{2}}}\,dt \\
&= (1+t)^{-\frac12}(f(m-1,n)+f(m,n-1)).
\end{align*}
Also, $f(n,0)=f(0,n)\lesssim (1+t)^{\frac12}, \log(1+t)$ or $1$ for $n=1, n=2$ or $n\geq3$. Using these two rules we obtain the given estimates. \qed
\end{proof}
As we have already dealt with \eqref{eq:exclusion-B} when $k=1$, we use the simple fact
\[ \sum_x h(x)^k \leq (\sum_x h(x))(\supl_x h(x))^{k-1}, \quad h\geq0, \]
to generalize to any $k$. However, we must show that \eqref{eq:exclusion-B} is bounded by a constant when we replace the sum by the supremum. When we use the same initial estimates as above, we get
\begin{align*}
& \supl_{x\in\bZ^d}\supl_{y\sim x}\intl_0^T\,dt\intl_0^t\,ds \sum_{z_1\in \bZ^d}\sum_{z_2\sim z_1} \left(\sum_{a\in A} p_s(a,z_1)\right)\left(\sum_{a\in A} p_s(a,z_2)\right)\\
&\cdot\abs{p_{t-s}(x,z_1) - p_{t-s}(x,z_2) - p_{t-s}(y,z_1) + p_{t-s}(y,z_2)}\norm{p_{t-s}(z_1,\cdot)-p_{t-s}(z_2,\cdot)}_{TVar},\\
\intertext{and by taking the sum over $z_1$ over the $p_{t-s}$ differences,}
&\lesssim \intl_0^T\,dt\intl_0^t\,ds\ (\abs{A}p_s(0,0))^2 (1+t-s)^{-\frac32} \\
&\lesssim \abs{A}^2\intl_0^T\,dt\intl_0^t\,ds\ (1+s)^{-d} (1+t-s)^{-\frac32} \lesssim \abs{A}^2\quad \text{ if } d\geq 2.
\end{align*}
Hence, finally, we have obtained the estimate
\[ \eqref{eq:exclusion-B} \lesssim \abs{A}^{2k}\alpha(T) .\]
Part \eqref{eq:exclusion-C} is treated in a similar way:
\begin{align*}
&\sum_{\substack{Z\subset \bZ^d\\ \abs{Z}=\abs{A}}} \bP_A(X_s = Z) \sum_{\substack{z_1,z_2\in Z\\z_1\neq z_2}}p(z_1,z_2) \suml_{\substack{z_3 \in Z\\z_3\neq z_1,z_2}} \abs{\rho^{\eta^{xy}}_{t-s}(z_3)-\rho^{\eta}_{t-s}(z_3)}(\rho^{\eta}_{t-s}(z_1) - \rho^{\eta}_{t-s}(z_2))^2 \\
&\lesssim \sum_{z_1}\sum_{z_2\sim z_1}\sum_{z_3} \prod_{i=1}^3\left(\sum_{a\in A}p_s(a,z_i)\right)
\abs{p_{t-s}(y,z_3)-p_{t-s}(x,z_3)}\norm{p_{t-s}(z_1,\cdot)-p_{t-s}(z_2,\cdot)}_{TVar}^2
\end{align*}
By using the fact that
\begin{align*}
&\sum_x \abs{p_{t-s}(x+j,z)-p_{t-s}(x,z)} = 2\norm{p_{t-s}(j,\cdot)-p_{t-s}(0,\cdot)}_{TVar}
\end{align*}
we can sum over $x$ to obtain another power of the total variation distance. Also,
\begin{align*}
\sum_{z_1}\sum_{z_2\sim z_1}\sum_{z_3} \prod_{i=1}^3\left(\sum_{a\in A}p_s(a,z_i)\right)&\leq 2d\abs{A}^3 p_{s}(0,0),
\end{align*}
hence we obtain the compound estimate
\begin{align*}
\abs{A}^3p_s(0,0) (1+t-s)^{-3/2},
\end{align*}
which after integrating over $s$ and $t$ is again of order $\alpha(T)$. When we take the supremum over $x$, we can instead take the sum over $z_3$ on the middle term. Hence we keep another $p_{s}(0,0)$ and we get and get
\begin{align*}
&\abs{A}^3p_s(0,0)^2 (1+t-s)^{-3/2},
\end{align*}
which after integration is of order $1$ if $d\geq2$. Hence,
\[ \eqref{eq:exclusion-C} \lesssim \abs{A}^{3k} \alpha(T).\]
Returning to the original question,
\begin{align*}
L\abs{\Phi_t}^k(\cdot,\eta)(\eta) \lesssim \abs{A}^k\alpha(T)+\abs{A}^{2k}\alpha(T)+\abs{A}^{3k}\alpha(T) \lesssim \abs{A}^{3k}\alpha(T),
\end{align*}
and after replacing $\lesssim$ with $\leq$,
\[ L\abs{\Phi_t}^k(\cdot,\eta)(\eta) \leq c_1 c_2^k\abs{A}^{3k}\alpha(T) .\]
Now that we have this estimate, Theorem \ref{thm:exponential expectation} gives us the estimate
\begin{align*}
\bE_\eta\exp\left(\int_0^T \lambda H_A(\eta_t)\,dt - \bE_\eta\int_0^T \lambda H_A(\eta_t)\,dt\right)\leq \exp\left(T\alpha(T)c_1\suml_{k=2}^\infty\frac{(c_2\lambda\abs{A}^3)^k}{k!}\right),
\end{align*}
where the constants $c_1$ and $c_2$ do not depend on $T$ or $A$.

\bibliography{../BibCollection}{}

\begin{thebibliography}{10}

\bibitem{ARRATIA:85}
R.~Arratia.
\newblock Symmetric exclusion processes: a comparison inequality and a large
  deviation result.
\newblock {\em Annals of Probability}, 13(1):53--61, 1985.

\bibitem{FERRARI:GALVES:LANDIM:00}
P.A. Ferrari, A.~Galves, and C.~Landim.
\newblock Rate of converge to equilibrium of symmetric simple exclusion
  processes.
\newblock {\em Markov processes and related fields}, 6:73--88, 2000.

\bibitem{CHAZOTTES:REDIG:09}
F.~Redig J.R.~Chazottes.
\newblock Concentration inequalities for markov processes via coupling.
\newblock {\em Electronic Journal of Probability}, 14:1162--1180, 2009.

\bibitem{WU:YAO:08}
N.~Yao L.~Wu.
\newblock Large deviation principles for markov processes via $\phi$-sobolev
  inequalities.
\newblock {\em Electronic Communications in Probability}, (13):10--23, 2008.

\bibitem{LANDIM:92}
C.~Landim.
\newblock Occupation time large deviations for the symmetric simple exclusion
  process.
\newblock {\em The Annals of Probability}, 20(1):206--231, 1992.

\bibitem{LEDOUX:01}
M.~Ledoux.
\newblock {\em The Concentration of Measure Phenomenon}.
\newblock American Mathematical Society, 2001.

\bibitem{LIGGETT:05}
T.M. Liggett.
\newblock {\em Interacting Particle Systems}.
\newblock Springer, 2005.

\bibitem{LINDVALL:92}
T.~Lindvall.
\newblock {\em Lectures on the coupling method}.
\newblock Wiley, 1992.

\bibitem{CATTIAUX:GUILLIN:08}
A.~Guillin P.~Cattiaux.
\newblock Deviation bounds for additive functionals of markov processes.
\newblock {\em ESAIM. Probability and Statistics}, 12:12--29, 2008.

\bibitem{YAO-FENG:FAN-JI:03}
Yao-Feng Ran and Fan-Ji Tian.
\newblock On the {R}osenthal's inequality for locally square integrable
  martingales.
\newblock {\em Stochastic processes and their applications}, 104:107--116,
  2003.

\bibitem{VILLANI:03}
C.~Villani.
\newblock {\em Topics in optimal transportation}.
\newblock American Mathematical Society, 2003.

\bibitem{VILLANI:09}
C.~Villani.
\newblock {\em Optimal transport. Old and new.}
\newblock Springer-Verlag, Berlin, 2009.

\end{thebibliography}
\bibliographystyle{plain}

\end{document}